\documentclass[12pt, reqno, a4paper]{amsart}

\usepackage{amssymb}
\usepackage{hyperref}

\usepackage[english]{babel}

\theoremstyle{plain}

\newtheorem*{corollary}{Corollary}
\newtheorem{lemma}{Lemma}
\newtheorem{theorem}{Theorem}
\newtheorem*{conjecture}{Conjecture}

\theoremstyle{remark}
\newtheorem*{remark}{Remark}

\theoremstyle{definition}

\newtheorem{example}{Example}

\DeclareMathOperator{\Id}{Id}

\DeclareMathOperator{\id}{id}
\DeclareMathOperator{\ch}{char}

\DeclareMathOperator{\ad}{ad}

\DeclareMathOperator{\tr}{tr}

\DeclareMathOperator{\End}{End}

\DeclareMathOperator{\PIexp}{PIexp}

\begin{document}

\title[Co-stability of radicals]{Co-stability of radicals and its applications to PI-theory}

\author{A.\,S.~Gordienko}
\address{Vrije Universiteit Brussel, Belgium}
\email{alexey.gordienko@vub.ac.be} 


\begin{abstract}
We prove that if $A$ is a finite dimensional associative $H$-comodule algebra
over a field $F$ for some involutory Hopf algebra $H$ not necessarily finite dimensional,
where either $\ch F = 0$ or $\ch F > \dim A$, then the Jacobson radical $J(A)$
is an $H$-subcomodule of $A$.
In particular, if $A$ is a finite dimensional associative algebra over such a field $F$, graded by any group, then the Jacobson radical $J(A)$ is a graded ideal of $A$.
Analogous results hold for nilpotent and solvable radicals of finite dimensional Lie algebras over a field of characteristic $0$.
We use the results obtained to prove the analog of Amitsur's conjecture for graded polynomial
identities of finite dimensional associative algebras over a field of characteristic $0$, graded
by any group.\end{abstract}


\thanks{Supported by Fonds Wetenschappelijk Onderzoek~--- Vlaanderen Pegasus Marie Curie post doctoral fellowship (Belgium).}

\maketitle

\section{Introduction}

If an algebra is endowed with an additional structure, e.g. grading, action of a group,
Lie or Hopf algebra, a natural question arises whether the radical is invariant with respect to this
structure. 
 
In 1984 M.~Cohen and S.~Montgomery~\cite{CohenMont} proved
that the Jacobson radical of an associative algebra
graded by a finite group $G$ is graded if $|G|^{-1}$
belongs to the base field. In 2001 V.~Linchenko~\cite{LinchenkoJH} proved
the stability of the Jacobson radical of a finite dimensional $H$-module associative algebra
for an involutory Hopf algebra~$H$.
This result was later generalized by V.~Linchenko, S.~Montgomery, L.W.~Small, and S.M.~Skryabin~\cite{LinMontSmall, Skryabin}.

In 2011, D.~Pagon, D.~Repov\v s, and M.V.~Zaicev~\cite[Proposition 3.3 and its proof]{PaReZai} proved that
the solvable radical of a finite dimensional Lie algebra over an algebraically closed field 
of characteristic $0$, graded by any group, is graded. In~2012, the author~\cite[Theorem~1]{ASGordienko4} proved that the solvable and the nilpotent radical of a finite dimensional $H$-(co)module Lie algebra over a field of characteristic $0$ are $H$-(co)invariant for any finite dimensional (co)semisimple Hopf algebra $H$. (In fact, exactly the same arguments can be used to derive the stability of the radicals 
in any finite dimensional $H$-module Lie algebra for an arbitrary involutory Hopf algebra $H$ not necessarily finite dimensional.)

Here we prove the co-stability of the Jacobson radical of a finite dimensional $H$-comodule associative algebra $A$ for an involutory Hopf algebra~$H$ under some restrictions on the characteristic of the base field (Theorem~\ref{TheoremRadicalHSubComod}). Unfortunately, if $H$ is infinite dimensional, we cannot derive this result from V.~Linchenko's one directly since the dual algebra $H^*$ is not necessarily a coalgebra.
However, using the fact that $A$ is finite dimensional, we provide a substitute for a comultiplication in $H^*$ (see Lemma~\ref{LemmaSubstituteComult}), which is enough for our purposes. 

As a consequence, we prove that the nilpotent and the solvable radical of a finite dimensional $H$-comodule Lie algebra
for an involutory Hopf algebra~$H$ over a field of characteristic $0$, are $H$-subcomodules (Theorem~\ref{TheoremLieRadicalHSubComod}).

As another consequence, we show that the radicals of finite dimensional algebras graded by an arbitrary group are graded. The latter result is used to prove the analog of Amitsur's conjecture for 
graded codimensions of finite dimensional algebras.

The original Amitsur's conjecture was proved in 1999 by
A.~Giambruno and M.V.~Zaicev~\cite[Theorem~6.5.2]{ZaiGia} for all associative PI-algebras.
  Alongside with ordinary polynomial
identities of algebras, graded, differential, $G$- and
$H$-identities are
important too~\cite{BahtGiaZai, BahtZaiGradedExp, BereleHopf}.
 Usually, to find such identities is easier
  than to find the ordinary ones. Furthermore, each of these types of identities completely determines  the ordinary polynomial identities. 
Therefore the question arises whether the conjecture
holds for graded codimensions, $G$-, $H$-codimensions,
and codimensions of polynomial identities with derivations.
The analog of Amitsur's conjecture
for codimensions of graded identities was proved in 2010--2011 by
E.~Aljadeff,  A.~Giambruno, and D.~La~Mattina~\cite{AljaGia, AljaGiaLa, GiaLa}
  for all associative PI-algebras graded by a finite group.
   In 2012, the author and M.V~Kotchetov~\cite[Theorem~13]{GordienkoKotchetov} proved
 the analog of Amitsur's conjecture  for graded polynomial
 identities of finite dimensional associative algebras
 graded by an Abelian group.
 
  Here we prove this conjecture for finite dimensional associative algebras graded by an arbitrary group
 (Theorem~\ref{TheoremMainGrAssoc}).

\section{$H$-comodule algebras}

Let $H$ be a Hopf algebra over a field $F$ with a comultiplication $\Delta \colon H \to H\otimes H$,
a counit $\varepsilon \colon H\to F$, and an antipode $S \colon H \to H$.
We say that $H$ is \textit{involutory} if $S^2=\id_H$.
Throughout the paper we use Sweedler's notation
$\Delta h = h_{(1)} \otimes h_{(2)}$ where $\Delta$ is the comultiplication
in $H$. We refer the reader to~\cite{Danara, Montgomery, Sweedler}
   for an account of Hopf algebras and algebras with Hopf algebra actions.

Recall that an algebra $A$ is a \textit{(right) $H$-comodule algebra} if 
$A$ is a right $H$-comodule and $\rho(ab)=a_{(0)}b_{(0)} \otimes a_{(1)}b_{(1)}$
for all $a, b \in A$ where $\rho \colon A \to A \otimes H$ is the comodule map.
(Here we use Sweedler's notation $\rho(a)=a_{(0)}\otimes a_{(1)}$
for $a\in A$.)

\begin{example}Let $A=\bigoplus_{g\in G}A^{(g)}$ be a graded algebra for some group $G$.
Then $A$ is an $FG$-comodule algebra for the group algebra $FG$
where $\rho(a^{(g)})=a^{(g)}\otimes g$ for all $a^{(g)}\in A^{(g)}$, $g\in G$.
\end{example}

 Note that an $H$-comodule algebra $A$ is a left $H^*$-module where $H^*$ is the algebra dual to the coalgebra $H$
 and $h^*a := h^*(a_{(1)})a_{(0)}$ for all $a\in A$ and $h^* \in H^*$.
 If $H$ is infinite dimensional, then it is not always possible to define the structure
 of a coalgebra on $H^*$ dual to the algebra $H$. However, we can provide some substitute for
 a comultiplication in $H^*$.

\begin{lemma}\label{LemmaSubstituteComult}
Suppose $A$ is a finite dimensional $H$-comodule algebra for some Hopf algebra $H$
over a field $F$. Choose a finite dimensional subspace $H_1$
such that $\rho(A)\subseteq A \otimes H_1$.
Then for every $h^*\in H^*$ there exist $s\in \mathbb N$, ${h_i^*}', {h_i^*}''\in H^*$, $1\leqslant i\leqslant s$, such that $$h^*(hq) = \sum_{i=1}^s{h_i^*}'(h){h_i^*}''(q) \text{
for all } h,q \in H_1.$$ In particular, $h^*(ab)=\sum_{i=1}^s ({h_i^*}'a)({h_i^*}''b)$ for all $a,b \in A$.
\end{lemma}
\begin{proof}
Consider the map $\Xi \colon H^* \to (H_1 \otimes H_1)^*$
where $\Xi(h^*)(h \otimes q) = h^*(hq)$ for $h^* \in H^*$,
$h,q \in H$. Using the natural identification $(H_1 \otimes H_1)^*=H_1^*\otimes H_1^*$,
we get $\Xi(h^*) = \sum_{i=1}^s {h_i^*}' \otimes {h_i^*}''$
for some $s\in \mathbb N$, ${h_i^*}', {h_i^*}''\in H_1^*$, $1\leqslant i\leqslant s$.
Extending linear functions from $H_1$ to $H$, we may assume that
${h_i^*}', {h_i^*}''\in H^*$.
Then $h^*(hq) = \sum_{i=1}^s{h_i^*}'(h){h_i^*}''(q)$
for all $h,q \in H_1$.
In particular, 
$$h^*(ab)=h^*(a_{(1)}b_{(1)})a_{(0)}b_{(0)} = 
\sum_{i=1}^s {h_i^*}'(a_{(1)}){h_i^*}''(b_{(1)}) a_{(0)}b_{(0)}
=\sum_{i=1}^s ({h_i^*}'a)({h_i^*}''b).$$ for all $a,b \in A$.
\end{proof}

\begin{remark}
Throughout the paper we choose $H_1$ as follows. First, we choose a finite dimensional subspace $H_2\subseteq H$ such that $\rho(A)\subseteq A \otimes H_2$.
Second, we choose a finite dimensional subspace $H_2 \subseteq H_3\subseteq H$ such that $\Delta(H_2)\subseteq H_3 \otimes H_3$. Finally, we define $H_1 := H_3 + SH_3$. Now we may assume that $a_{(1)}, a_{(2)}, Sa_{(1)},
Sa_{(2)} \in H_1$ for all $a\in A$  in expressions like $$(\id_A \otimes \Delta)\rho(a)=a_{(0)}\otimes a_{(1)}\otimes a_{(2)}.$$
\end{remark}

  \begin{lemma}\label{LemmaHComoduleIdeal}
  Let $A$ be a finite dimensional $H$-comodule algebra over a field $F$
for some Hopf algebra $H$ with a bijective antipode. If $I$ is an ideal of $A$, then $H^*I$
is an ideal of $A$ too.
  \end{lemma}
  \begin{proof} Let $a \in I$, $b\in A$, $h^*\in H^*$. Then
  $$(h^*a)b=h^*(a_{(1)})a_{(0)}b=h^*(\varepsilon(b_{(1)})a_{(1)})a_{(0)}b_{(0)}
  =h^*(a_{(1)}b_{(1)}Sb_{(2)})a_{(0)}b_{(0)}
  =$$ \begin{equation}\label{EqMoveHFirst}\sum_i {h^*_i}'(a_{(1)}b_{(1)}){h^*_i}''(Sb_{(2)})a_{(0)}b_{(0)}
  =\sum_i {h^*_i}'(a (S^*{h^*_i}'') b) \in H^*I.\end{equation}
  Similarly,
  $$b(h^*a)=h^*(a_{(1)})ba_{(0)}=h^*(\varepsilon(b_{(1)})a_{(1)})b_{(0)}a_{(0)}
  =h^*((S^{-1}b_{(2)})b_{(1)} a_{(1)})b_{(0)}a_{(0)}
  =$$ \begin{equation}\label{EqMoveHSecond}\sum_i {h^*_i}'(S^{-1}b_{(2)}){h^*_i}''(b_{(1)}a_{(1)})b_{(0)}a_{(0)}
  =\sum_i {h^*_i}''((((S^{-1})^*{h^*_i}') b) a) \in H^*I.\end{equation}
  \end{proof}
  
  In addition, we need the Wedderburn~--- Artin theorem for $H$-comodule algebras:
  
  \begin{lemma}\label{LemmaWedderburnHcomod}
  Let $B$ be a finite dimensional semisimple associative $H$-comodule algebra over a field $F$
  for some Hopf algebra $H$ with a bijective antipode. Then $B = B_1 \oplus B_2 \oplus \ldots \oplus B_s$
  (direct sum of $H$-coinvariant ideals) for some $H$-simple algebras $B_i$.
  \end{lemma}
  \begin{proof}
  By the original Wedderburn~--- Artin theorem, $B=A_1 \oplus \ldots \oplus A_s$ (direct sum of ideals) 
where $A_i$ are simple algebras not necessarily $H$-subcomodules.
Let $B_1$ be a minimal ideal of $A$ that is an $H$-subcomodule.
 Then $B_1 = A_{i_1} \oplus \ldots \oplus A_{i_k}$
for some $i_1, i_2, \ldots, i_k \in \lbrace 1, 2, \ldots, s\rbrace$. 
Consider $\tilde B_1 = \lbrace a \in A \mid ab=ba=0 \text{ for all } b \in B_1 \rbrace$.
Then $\tilde B_1$ equals the sum of all $A_j$, $j \notin \lbrace 
i_1, i_2, \ldots, i_k\rbrace$, and $A = B_1 \oplus \tilde B_1$.
 We claim that $\tilde B_1$ is an $H$-subcomodule. It is sufficient to prove that for every $h^* \in H^*$, $a \in \tilde B_1$,
 $b \in B_1$, we have $h^*(a_{(1)})a_{(0)} b = h^*(a_{(1)}) b a_{(0)} =0$.
   However, this follows from~(\ref{EqMoveHFirst}) and~(\ref{EqMoveHSecond}).
  Hence $\tilde B_1$ is an $H$-subcomodule and the inductive argument finishes the proof.
  \end{proof}

\section{Co-stability of the Jacobson radical}

Note that if $A$ is a finite dimensional $H$-comodule algebra for a Hopf algebra $H$, then 
$\End_F(A)$ is an $H$-comodule algebra where $\psi_{(0)}a \otimes \psi_{(1)}
= \psi(a_{(0)})_{(0)} \otimes \psi(a_{(0)})_{(1)}(Sa_{(1)})$
for $\psi \in \End_F(A)$ and $a\in A$.
 Moreover, 
 $$(h^*\psi)a = h^*(\psi(a_{(0)})_{(1)}(Sa_{(1)})) \psi(a_{(0)})_{(0)}
 = \sum_i {h^*_i}'\psi((S^*{h^*_i}'') a)$$
for $h^*\in H^*$, $\psi \in \End_F(A)$, and $a\in A$. 
 Hence $h^*\psi = \sum_i \zeta({h^*_i}') \psi \zeta(S^*{h^*_i}'')$
 for all $h^* \in H^*$ and $\psi \in \End_F(A)$
 where $\zeta \colon H^* \to \End_F(A)$ is the map corresponding to the $H^*$-module
 structure on $A$.

 \begin{lemma}\label{LemmaTraceHRad}
Let $A$ be a finite dimensional $H$-comodule algebra over a field $F$
for some Hopf algebra $H$ with $S^2=\id_H$.
 Consider the left regular representation $\Phi \colon A \to \End_F(A)$ where $\Phi(a)b := ab$.
 Then $$\tr(\Phi(h^* a))=h^*(1) \tr(\Phi(a))\text { for all } h^{*}\in H^{*} \text{ and } a \in A.$$
 \end{lemma}
 \begin{proof} Note that $\Phi$ is a homomorphism of $H$-comodules.
 Therefore, $\Phi$ is a homomorphism of $H^*$-modules.
 Moreover,
 $$\tr(\Phi(h^* a))= \sum_i \tr\zeta({h^*_i}') \Phi(a) \zeta(S^*{h^*_i}'')
 = \sum_i \tr(\zeta((S^*{h^*_i}''){h^*_i}') \Phi(a)).$$
 
Note that  $$\sum_i \zeta((S^*{h^*_i}''){h^*_i}') b = 
\sum_i(S^*{h^*_i}'')({h^*_i}'(b_{(1)})b_{(0)})=$$ $$
\sum_i({h^*_i}'')(Sb_{(1)}){h^*_i}'(b_{(2)})b_{(0)}
=h^*(b_{(2)} Sb_{(1)}) b_{(0)} =h^*(1)b$$ for all $h^*\in H^*$
and $b\in A$.
Hence $\tr(\Phi(h^* a)) = h^*(1) \tr(\Phi(a))$.
  \end{proof}

\begin{theorem}\label{TheoremRadicalHSubComod}
Let $A$ be a finite dimensional associative $H$-comodule algebra over a field $F$
for some Hopf algebra $H$ with $S^2=\id_H$.
Suppose that either $\ch F = 0$ or $\ch F>\dim A$. Then the Jacobson radical $J:=J(A)$ is an $H$-subcomodule of $A$.
\end{theorem}
\begin{corollary}
Let $A$ be a finite dimensional associative algebra over a field $F$ graded by any group $G$.
Suppose that either $\ch F = 0$ or $\ch F > \dim A$. Then the Jacobson radical $J:=J(A)$ is a graded ideal of $A$.
\end{corollary}
Combining this with~\cite[Corollary~2.8]{SteVanOyst}, we get
the graded Wedderburn~--- Mal'cev theorem:
\begin{corollary}
Let $A$ be a finite dimensional associative algebra over a field $F$ graded by any group $G$.
Suppose that either $\ch F = 0$ or $\ch F > \dim A$ and $A/J(A)$ is separable. Then there
exists a maximal semisimple subalgebra $B\subseteq A$
such that $A=B\oplus J(A)$ (direct sum of graded subspaces).
\end{corollary}

\begin{proof}[Proof of Theorem~\ref{TheoremRadicalHSubComod}.]
 Note that $J_0 := H^* J \supseteq J$ is an $H$-subcomodule. By Lemma~\ref{LemmaHComoduleIdeal},
 $J_0$ is an ideal. Therefore,
 it is sufficient to prove that  $J_0$ is a nil-ideal.
 
 Let $a_1, \ldots, a_m \in J$, $h^*_1, \ldots, h^*_m \in H^*$.
 Note that $J$ is nilpotent and $\tr(\Phi(a))=0$ for all $a \in J$. By~(\ref{EqMoveHFirst})
 and Lemma~\ref{LemmaTraceHRad}, $$\tr\left(\Phi((h^*_1 a_1)\ldots (h^*_m a_m))\right)
 = \sum_i\tr\left(\Phi\left({h^*_{1i}}'' (a_1 (S^*{h^*_{1i}}'')((h^*_2 a_2)\ldots (h^*_m a_m)))\right)\right)
 =$$ $$ \sum_i {h^*_{1i}}''(1)\tr\left(\Phi\left(a_1 (S^*{h^*_{1i}}'')((h^*_2 a_2)\ldots (h^*_m a_m))\right)\right)=0$$
 since $a_1 (S^*{h^*_{1i}}'')((h^*_2 a_2)\ldots (h^*_m a_m)) \in J$. 
 
 In particular, $\tr(\Phi(a)^k)=0$ for all $a\in J_0$ and $k\in\mathbb N$. Since either $\ch F = 0$
 or $\ch F > \dim A$, $\Phi(a)$ is a nilpotent operator on $A$ and $J_0$ is a nil-ideal. Hence $J = J_0$.
\end{proof}

\section{Co-stability of radicals in Lie algebras}\label{SectionLieRadical}

Analogous results hold for Lie algebras:

\begin{theorem}
\label{TheoremLieRadicalHSubComod}
Let $L$ be a finite dimensional $H$-comodule Lie algebra
over a field $F$ of characteristic $0$ for some Hopf algebra $H$ with $S^2=\id_H$.
Then the solvable radical $R$ and the nilpotent radical $N$ of $L$ are $H$-subcomodules.
\end{theorem}
\begin{proof}
Consider the adjoint representation $\ad \colon L \to \mathfrak{gl}(L)$
where $(\ad a)b:=[a,b]$.
Then $\ad$ is a homomorphism of $H$-comodules and $H^*$-modules.
Denote by $A$ the associative subalgebra of $\End_F(L)$ generated by $(\ad L)$.
Applying Lemma~\ref{LemmaSubstituteComult} to $\End_F(L)$, we obtain that $A$ is an $H^*$-submodule. Therefore, $A$ is an $H$-subcomodule.
By Lemma~\ref{LemmaHComoduleIdeal}, $H^*N$ and $H^*R$ are ideals of $L$.
Theorem~\ref{TheoremRadicalHSubComod} and \cite[Lemma~1]{ASGordienko4} imply $(\ad (H^*N)) \subseteq H^*J(A)=J(A)$. Thus $H^*N$ is nilpotent and $H^*N = N$.

By \cite[Proposition 2.1.7]{GotoGrosshans},
 $[L, R] \subseteq N$. Together with~(\ref{EqMoveHFirst})
 this implies $$[H^*R, H^*R] \subseteq [H^*R, L] \subseteq H^* [R, H^*L] \subseteq H^*[R,L] \subseteq H^*N = N.$$
 Thus $H^*R$ is solvable and $H^*R = R$.
\end{proof}

As an immediate consequence, we get D.~Pagon, D.~Repov\v s, and M.V.~Zaicev's result:

\begin{corollary}
Let $L$ be a finite dimensional Lie algebra
over a field $F$ of characteristic $0$, graded by an arbitrary group.
Then the solvable radical $R$ and the nilpotent radical $N$ of $L$ are graded ideals.
\end{corollary}

Combining this with~\cite[Theorem~4]{ASGordienko4}, we obtain the graded Levi theorem:
\begin{corollary}
Let $L$ be a finite dimensional Lie algebra over a field $F$ of characteristic $0$,
graded by an arbitrary group $G$.
Then there exists a maximal semisimple subalgebra $B$ in $L$ such that
$L=B\oplus R$ (direct sum of graded subspaces).
\end{corollary}

\section{Graded polynomial identities and their codimensions}\label{SectionGraded}

Let $G$ be a group and let $F$ be a field. Denote by $F\langle X^{\mathrm{gr}} \rangle $ the free $G$-graded associative  algebra over $F$ on the countable set $$X^{\mathrm{gr}}:=\bigcup_{g \in G}X^{(g)},$$ $X^{(g)} = \{ x^{(g)}_1,
x^{(g)}_2, \ldots \}$,  i.e. the algebra of polynomials
 in non-commuting variables from $X^{\mathrm{gr}}$.
  The indeterminates from $X^{(g)}$ are said to be homogeneous of degree
$g$. The $G$-degree of a monomial $x^{(g_1)}_{i_1} \dots x^{(g_t)}_{i_t} \in F\langle
 X^{\mathrm{gr}} \rangle $ is defined to
be $g_1 g_2 \dots g_t$, as opposed to its total degree, which is defined to be $t$. Denote by
$F\langle
 X^{\mathrm{gr}} \rangle^{(g)}$ the subspace of the algebra $F\langle
 X^{\mathrm{gr}} \rangle$ spanned
 by all the monomials having
$G$-degree $g$. Notice that $$F\langle
 X^{\mathrm{gr}} \rangle^{(g)} F\langle
 X^{\mathrm{gr}} \rangle^{(h)} \subseteq F\langle
 X^{\mathrm{gr}} \rangle^{(gh)},$$ for every $g, h \in G$. It follows that
$$F\langle
 X^{\mathrm{gr}} \rangle =\bigoplus_{g\in G} F\langle
 X^{\mathrm{gr}} \rangle^{(g)}$$ is a $G$-grading.
  Let $f=f(x^{(g_1)}_{i_1}, \dots, x^{(g_t)}_{i_t}) \in F\langle
 X^{\mathrm{gr}} \rangle$.
We say that $f$ is
a \textit{graded polynomial identity} of
 a $G$-graded algebra $A=\bigoplus_{g\in G}
A^{(g)}$
and write $f\equiv 0$
if $f(a^{(g_1)}_{i_1}, \dots, a^{(g_t)}_{i_t})=0$
for all $a^{(g_j)}_{i_j} \in A^{(g_j)}$, $1 \leqslant j \leqslant t$.
  The set $\Id^{\mathrm{gr}}(A)$ of graded polynomial identities of
   $A$ is
a graded ideal of $F\langle
 X^{\mathrm{gr}} \rangle$.
The case of ordinary polynomial identities is included
for the trivial group $G=\lbrace e \rbrace$.

\begin{example}\label{ExampleIdGr}
 Let $G=\mathbb Z_2 = \lbrace \bar 0, \bar 1 \rbrace$,
$M_2(F)=M_2(F)^{(\bar 0)}\oplus M_2(F)^{(\bar 1)}$
where $M_2(F)^{(\bar 0)}=\left(
\begin{array}{cc}
F & 0 \\
0 & F
\end{array}
 \right)$ and $M_2(F)^{(\bar 1)}=\left(
\begin{array}{cc}
0 & F \\
F & 0
\end{array}
 \right)$. Then  $x^{(\bar 0)} y^{(\bar 0)} - y^{(\bar 0)} x^{(\bar 0)}
\in \Id^{\mathrm{gr}}(M_2(F))$.
\end{example}

Let
$P^{\mathrm{gr}}_n := \langle x^{(g_1)}_{\sigma(1)}
x^{(g_2)}_{\sigma(2)}\ldots x^{(g_n)}_{\sigma(n)}
\mid g_i \in G, \sigma\in S_n \rangle_F \subset F \langle X^{\mathrm{gr}} \rangle$, $n \in \mathbb N$.
Then the number $$c^{\mathrm{gr}}_n(A):=\dim\left(\frac{P^{\mathrm{gr}}_n}{P^{\mathrm{gr}}_n \cap \Id^{\mathrm{gr}}(A)}\right)$$
is called the $n$th \textit{codimension of graded polynomial identities}
or the $n$th \textit{graded codimension} of $A$.

The analog of Amitsur's conjecture for graded codimensions can be formulated
as follows.

\begin{conjecture} There exists
 $\PIexp^{\mathrm{gr}}(A):=\lim\limits_{n\to\infty} \sqrt[n]{c^\mathrm{gr}_n(A)} \in \mathbb Z_+$.
\end{conjecture}

\begin{theorem}\label{TheoremMainGrAssoc}
Let $A$ be a finite dimensional non-nilpotent associative algebra
over a field $F$ of characteristic $0$, graded by any group $G$. Then
there exist constants $C_1, C_2 > 0$, $r_1, r_2 \in \mathbb R$, $d \in \mathbb N$
such that $C_1 n^{r_1} d^n \leqslant c^{\mathrm{gr}}_n(A) \leqslant C_2 n^{r_2} d^n$
for all $n \in \mathbb N$.
\end{theorem}
\begin{corollary}
The above analog of Amitsur's conjecture holds for such codimensions.
\end{corollary}
\begin{remark}
If $A$ is nilpotent, i.e. $x_1 \ldots x_p \equiv 0$ for some $p\in\mathbb N$, then  $P^{\mathrm{gr}}_n \subseteq \Id^{\mathrm{gr}}(A)$ and $c^{\mathrm{gr}}_n(A)=0$ for all $n \geqslant p$.
\end{remark}

Theorem~\ref{TheoremMainGrAssoc} will be proved in Section~\ref{SectionProofTheoremMainHAssoc}.

\section{Polynomial $H$-identities and their codimensions}\label{SectionH}

In the proof of Theorem~\ref{TheoremMainGrAssoc}, we use the notion of
generalized Hopf action~\cite[Section~3]{BereleHopf}.

Let $H$ be an associative algebra with $1$ over a field $F$.
We say that an associative algebra $A$ is an algebra with a \textit{generalized $H$-action}
if $A$ is endowed with a homomorphism $H \to \End_F(A)$
and for every $h \in H$ there exist $h'_i, h''_i, h'''_i, h''''_i \in H$
such that 
\begin{equation}\label{EqGeneralizedHopf}
h(ab)=\sum_i\bigl((h'_i a)(h''_i b) + (h'''_i b)(h''''_i a)\bigr) \text{ for all } a,b \in A.
\end{equation}

\begin{example}
Let $A$ be a \textit{(left) $H$-module algebra} over a field $F$
where $H$ is a Hopf algebra, i.e. 
$A$ is endowed with a homomorphism $H \to \End_F(A)$ such that
$h(ab)=(h_{(1)}a)(h_{(2)}b)$
for all $h \in H$, $a,b \in A$. Then $A$ is an algebra with a generalized $H$-action.
\end{example}

\begin{example}
Let $A$ be a finite dimensional $H$-comodule algebra over a field $F$
where $H$ is a Hopf algebra. Then by Lemma~\ref{LemmaSubstituteComult},
 $A$ is an algebra with a generalized $H^*$-action.
\end{example}

\begin{example}\label{ExampleGrHGen}
Let $A=\bigoplus_{g\in G}
A^{(g)}$ be an algebra graded by a group $G$ and let
$(FG)^*$ be the algebra dual to the group coalgebra $FG$. In other words, $(FG)^*$ is the algebra
of functions $G \to F$ with the pointwise multiplication. 
Then $A$ has the following natural $(FG)^*$-action: $ha^{(g)} = h(g)a^{(g)}$ for all $g\in G$,
$a^{(g)} \in A^{(g)}$, and $h\in (FG)^*$. If $A$ is finite dimensional, 
$A=\bigoplus_{k=1}^m
A^{(\gamma_k)}$ for some $m \in \mathbb Z_+$, $\gamma_1, \ldots, \gamma_m \in G$.
Therefore, $$h(ab)=\sum_{j,k=1}^m h(\gamma_j \gamma_k) (\pi_{\gamma_j}a)(\pi_{\gamma_k}b)=\sum_{j,k=1}^m h(\gamma_j \gamma_k) (h_{\gamma_j}a)(h_{\gamma_k}b) \text{ for all } 
h\in (FG)^*,\ a,b \in A,$$
where $h_g \in (FG)^*$ are delta functions: $h_{g_1}(g_2)=\left\lbrace\begin{array}{cc} 1, & g_1 = g_2, \\ 0, & g_1 \ne g_2,\end{array}\right.$ 
and $\pi_g \colon A \to A^{(g)}$ are the natural projections with the kernels $\bigoplus_{\substack{t\in G, \\ t\ne g}} A^{(t)}$.
Hence $A$ is an algebra with a generalized $(FG)^*$-action.
\end{example}

Choose a basis $(\gamma_\beta)_{\beta \in \Lambda}$ in $H$ and
denote by $F \langle X | H \rangle$ 
the free associative algebra over $F$ with free formal
 generators $x_i^{\gamma_\beta}$, $\beta \in \Lambda$, $i \in \mathbb N$.
 Let $x_i^h := \sum_{\beta \in \Lambda} \alpha_\beta x_i^{\gamma_\beta}$
 for $h= \sum_{\beta \in \Lambda} \alpha_\beta \gamma_\beta$, $\alpha_\beta \in F$,
 where only finite number of $\alpha_\beta$ are nonzero.
Here $X := \lbrace x_1, x_2, x_3, \ldots \rbrace$, $x_j := x_j^1$, $1 \in H$.
 We refer to the elements
 of $F\langle X | H \rangle$ as $H$-polynomials.
Note that here we do not consider any $H$-action on $F \langle X | H \rangle$.

Let $A$ be an associative algebra with a generalized $H$-action.
Any map $\psi \colon X \to A$ has a unique homomorphic extension $\bar\psi
\colon F \langle X | H \rangle \to A$ such that $\bar\psi(x_i^h)=h\psi(x_i)$
for all $i \in \mathbb N$ and $h \in H$.
 An $H$-polynomial
 $f \in F\langle X | H \rangle$
 is an \textit{$H$-identity} of $A$ if $\bar\psi(f)=0$
for all maps $\psi \colon X \to A$. In other words, $f(x_1, x_2, \ldots, x_n)$
 is an $H$-identity of $A$
if and only if $f(a_1, a_2, \ldots, a_n)=0$ for any $a_i \in A$.
 In this case we write $f \equiv 0$.
The set $\Id^{H}(A)$ of all $H$-identities
of $A$ is an ideal of $F\langle X | H \rangle$.
Note that our definition of $F\langle X | H \rangle$
depends on the choice of the basis $(\gamma_\beta)_{\beta \in \Lambda}$ in $H$.
However such algebras can be identified in the natural way,
and $\Id^H(A)$ is the same.

Denote by $P^H_n$ the space of all multilinear $H$-polynomials
in $x_1, \ldots, x_n$, $n\in\mathbb N$, i.e.
$$P^{H}_n = \langle x^{h_1}_{\sigma(1)}
x^{h_2}_{\sigma(2)}\ldots x^{h_n}_{\sigma(n)}
\mid h_i \in H, \sigma\in S_n \rangle_F \subset F \langle X | H \rangle.$$
Then the number $c^H_n(A):=\dim\left(\frac{P^H_n}{P^H_n \cap \Id^H(A)}\right)$
is called the $n$th \textit{codimension of polynomial $H$-identities}
or the $n$th \textit{$H$-codimension} of $A$.

We need the following theorem:

\begin{theorem}[{\cite[Theorem~1]{ASGordienko8}}]\label{TheoremMainHAssoc} Let $A$ be a finite dimensional non-nilpotent
associative algebra with a generalized $H$-action
where $H$ is an associative algebra with $1$ over an algebraically closed field $F$ of characteristic $0$.
Suppose that the Jacobson radical $J(A)$ is an $H$-submodule
and $$A/J(A) = B_1 \oplus \ldots \oplus B_q \text{ (direct sum of $H$-invariant ideals)}$$
for some $H$-simple algebras $B_i$.
Then there exist constants $C_1, C_2 > 0$, $r_1, r_2 \in \mathbb R$, and $d\in\mathbb N$ such that $$C_1 n^{r_1} d^n \leqslant c^{H}_n(A) \leqslant C_2 n^{r_2} d^n\text{ for all }n \in \mathbb N.$$
\end{theorem}

\section{Proof of Theorem~\ref{TheoremMainGrAssoc}. One example}\label{SectionProofTheoremMainHAssoc}

In order to apply Theorem~\ref{TheoremMainHAssoc}, we need the following lemma:

\begin{lemma}\label{LemmaHGenGrCodimEqual}
Let $A=\bigoplus_{g\in G}
A^{(g)}$ be a finite dimensional associative algebra over a field $F$ graded by any group $G$.
 Consider the corresponding generalized $H$-action for $H=(FG)^*$ (see Example~\ref{ExampleGrHGen}).
 Then $c_n^{\mathrm{gr}}(A)=c_n^{H}(A)$ for all $n\in \mathbb N$.
\end{lemma}
\begin{proof}
Again, let $\lbrace \gamma_1, \ldots, \gamma_m \rbrace := \lbrace g\in G \mid A^{(g)}\ne 0\rbrace$.

Define the homomorphism of algebras $\xi \colon F \langle X | H \rangle \to F \langle X^{\mathrm{gr}}\rangle$ 
by the formula $\xi(x_k^h)=\sum_{i=1}^m h(\gamma_i) x_k^{(\gamma_i)}$, $h \in H$, $k\in\mathbb N$. Note that $\xi(\Id^H(A))\subseteq \Id^{\mathrm{gr}}(A)$
since for any homomorphism $\psi \colon F \langle X^{\mathrm{gr}}\rangle \to A$ of graded algebras
we have $\psi(\xi(x_i^h))=h\psi(\xi(x_i))$ and if $f\in \Id^{H}(A)$, then $\psi(\xi(f))=0$.
Hence we can define $\tilde\xi \colon F \langle X | H \rangle/\Id^{H}(A) \to F \langle X^{\mathrm{gr}}\rangle/\Id^{\mathrm{gr}}(A)$.

Let $h_{\gamma_1}, \ldots, h_{\gamma_m} \in H$ be the delta functions from Example~\ref{ExampleGrHGen}.
 Then $h_{\gamma_i}a \in A^{(\gamma_i)}$
is the $\gamma_i$-component of $a$ for all $a \in A$ and $1 \leqslant i\leqslant m$.
In particular, \begin{equation}\label{EqxhiequivAssoc}
x^h - \sum_{i=1}^m h(\gamma_i)x^{h_{\gamma_i}} \in \Id^{H} (A) \text { for all }h\in H.\end{equation}

 We define the homomorphism of algebras
$\eta \colon F \langle X^{\mathrm{gr}}\rangle \to  F \langle X | H \rangle$ by the formula $\eta(x^{(\gamma_i)}_j)=x^{h_{\gamma_i}}_j$
for all $1\leqslant i\leqslant m$, $j\in\mathbb N$, and $\eta(x^{(g)}_j)=0$
for $g\notin \lbrace \gamma_1, \ldots, \gamma_m \rbrace$. Note that
$\eta(\Id^{\mathrm{gr}}(A)) \subseteq \Id^{H}(A)$.
Indeed, if $\psi \colon F \langle X | H \rangle \to A$ is a homomorphism of algebras
such that  $\psi(x_i^h)=h\psi(x_i)$ for all $h\in H$ and $i\in\mathbb N$,
then $\psi(\eta(x^{(\gamma_i)}_j)) = {h_{\gamma_i}} \psi(x_j) \in A^{(\gamma_i)}$
for any choice of $\psi(x_j)\in A$. Hence $\psi\eta \colon F \langle X^{\mathrm{gr}}\rangle \to A$
is a graded homomorphism, $\psi(\eta(\Id^{\mathrm{gr}}(A)))=0$, and $\eta(\Id^{\mathrm{gr}}(A)) \subseteq \Id^{H}(A)$.
Thus we can define $\tilde\eta \colon F \langle X^{\mathrm{gr}}\rangle/\Id^{\mathrm{gr}}(A) \to   F \langle X | H \rangle/\Id^{H}(A)$. 

Denote by $\bar f$ the image of a polynomial $f$ in a factor space.
Then $$\tilde\xi\tilde\eta(\bar x^{(\gamma_i)}_j) = \sum_{k=1}^m h_{\gamma_i}(\gamma_k) \bar x^{(\gamma_k)}_j
=\bar x^{(\gamma_i)}_j
\text{ for all }1 \leqslant i \leqslant m,\ j\in\mathbb N.$$ Since $x^{(g)}_j \in \Id^{\mathrm{gr}}(A)$
for all $g\notin \lbrace \gamma_1, \ldots, \gamma_m \rbrace$, the map $\tilde\xi\tilde\eta$ coincides with
the identity map on the generators. Hence  $\tilde\xi\tilde\eta = \id_{F \langle X^{\mathrm{gr}}\rangle/\Id^{\mathrm{gr}}(A)}$. By~(\ref{EqxhiequivAssoc}), $$\tilde\eta\tilde\xi(\bar x_j^h)=
\tilde\eta\left( \sum_{i=1}^m h(\gamma_i) \bar x_j^{(\gamma_i)} \right)=
 \sum_{i=1}^m h(\gamma_i)\bar x_j^{h_{\gamma_i}} = \bar x_j^h \text{ for all } j \in\mathbb N,\ h \in H.$$
Similarly, we have $\tilde\eta\tilde\xi = \id_{ F \langle X | H \rangle/\Id^{H}}$.
Hence $$ F \langle X | H \rangle/\Id^{H}(A) \cong F \langle X^{\mathrm{gr}}\rangle/\Id^{\mathrm{gr}}(A).$$
In particular, $\frac{P^{H}_n}{P^{H}_n \cap \Id^{H}(A)}
\cong \frac{P^{\mathrm{gr}}_n}{P^{\mathrm{gr}}_n \cap \Id^{\mathrm{gr}}(A)}$
and $c^{\mathrm{gr}}_n(A)=c^{H}(A)$ for all $n\in \mathbb N$.
\end{proof}

\begin{proof}[Proof of Theorem~\ref{TheoremMainGrAssoc}]
By Example~\ref{ExampleGrHGen}, $A$ is an algebra with a generalized $H$-action
for $H=(FG)^*$. By Lemma~\ref{LemmaHGenGrCodimEqual}, graded codimensions of $A$ coincide
with its $H$-codimensions. By Theorem~\ref{TheoremRadicalHSubComod},
$J(A)$ is a graded ideal and, therefore, an $H$-submodule.
By Lemma~\ref{LemmaWedderburnHcomod}, $A/J(A)$ is a direct sum of graded simple algebras.
 Now we use Theorem~\ref{TheoremMainHAssoc}.
\end{proof}

\begin{example}
Let $G$ be the free group with free generators $a_1,\ldots, a_\ell$, $\ell \geqslant 2$, let $F$ be a field of characteristic $0$, and let $k\in\mathbb N$. Recall that the group algebra $FG$ has the natural $G$-grading $$FG=\bigoplus_{g\in G} (FG)^{(g)} \text{ where }(FG)^{(g)} = Fg,\ g\in G.$$ Consider the subalgebra $A$ of $FG$ generated by $1, a_1, \ldots, a_\ell$ and the ideal $I_k \subset A$ generated by all products $a_{i_1}\ldots a_{i_n}$ of length $n\geqslant k$.
Note that both $A$ and $I_k$ are graded and $A/I_k$ is a finite dimensional algebra graded by an infinite non-Abelian group $G$. We claim that there exist $r_1, r_2 \in \mathbb R$, $C_1, C_2 > 0$
such that $C_1 n^{r_1}\leqslant c_n^{\mathrm{gr}}(A/I_k) \leqslant C_2 n^{r_2}$ for all $n\in\mathbb N$.
\end{example}
\begin{proof}
Note that $A/I_k= (F1) \oplus J$ (direct sum of subspaces) where $J$ is the ideal
generated by the images of $a_i$ in $A/I_k$. Moreover, $J$ is the Jacobson radical of $A/I_k$ since $J$ is nilpotent. Using Theorem~\ref{TheoremMainGrAssoc}, Lemma~\ref{LemmaHGenGrCodimEqual}, and the formula in~\cite[Theorem~1]{ASGordienko8},
we get $$\PIexp^{\mathrm{gr}}(A/I_k)=\dim((A/I_k)/J)=1.$$
\end{proof}

\section*{Acknowledgements}

This work started while I was an AARMS postdoctoral fellow at Memorial University of Newfoundland, whose faculty and staff I would like to thank for hospitality. I am grateful to Yuri Bahturin, who suggested that I study polynomial $H$-identities, and to Eric Jespers for helpful discussions.

\end{document}